\documentclass[a4paper,10pt]{article}

\usepackage[english]{babel}
\usepackage{amssymb,amsmath,amsthm}
\usepackage{euscript}
\usepackage{hyperref}
\usepackage{graphicx}
\usepackage{amsfonts}

\newtheorem{Th}{Theorem}
\newtheorem{lem}{Lemma}
\newtheorem{cor}{Corollary}
\newtheorem{prop}{Proposition}
\newtheorem{defi}{Definition}
\newtheorem{rem}{Remark}

\renewcommand{\phi}{\varphi}

\newcommand{\BMO}{\mathrm{BMO}}
\newcommand{\eps}{\varepsilon}

\newcommand{\BB}{\boldsymbol{B}_3}
\newcommand{\B}{\boldsymbol{B}}
\newcommand{\Bg}{\mathfrak{B}}
\newcommand{\diag}{b}

\newcommand{\conv}{\textup{conv}}
\newcommand{\intt}{\textup{int}}
\newcommand{\Sg}{\textup{Sg}}
\newcommand{\I}{\textup{I}}

\newcommand{\Omg}{{\Omega}_3}

\begin{document}
\selectlanguage{english}
\title{Bellman VS Beurling: sharp estimates of uniform convexity for~$L^p$ spaces}

\author{P. Ivanisvili\thanks{This paper was completed during a visit of the first author to the Hausdorff Research Institute for Mathematics (HIM) in the framework of the Trimester Program ``Harmonic Analysis and Partial Differential Equations''. He thanks HIM for the hospitality.} \and D. M. Stolyarov\thanks{supported by the Chebyshev Laboratory  (Department of Mathematics and Mechanics, St. Petersburg State University) RF Government grant 11.G34.31.0026, by JSC ``Gazprom Neft'', and by RFBR grant no. 11-01-00526.}
 \and P. B. Zatitskiy\thanks{supported by  the Chebyshev Laboratory  (Department of Mathematics and Mechanics, St. Petersburg State University) RF Government grant 11.G34.31.0026, by JSC ``Gazprom Neft'', by President of Russia grant for young researchers MK-6133.2013.1, by the RFBR (grant 13-01-12422 ofi\_m2, 14-01-00373\_A), and by SPbSU (thematic project 6.38.223.2014).}}

\maketitle
\begin{abstract}
We obtain the classical Hanner inequalities by the Bellman function method. These inequalities give sharp estimates for the moduli of convexity of Lebesgue spaces. Easy ideas from differential geometry help us to find the Bellman function using neither ``magic guesses'' nor calculations. 
\end{abstract}
\tableofcontents
\section{Classical results}
In 1936 Clarkson~\cite{Cl} introduced the notion of uniform convexity for normed spaces.
\begin{defi}
A normed space~$(X, \|\cdot\|)$ is said to be uniformly convex if for any~$\eps>0$ there exists~$\delta>0$ such that if~$x,y \in X$\textup, $\|x\|=\|y\|=1$\textup, and $\|x-y\|\geqslant \eps,$ then $\left\|\frac{x+y}{2} \right\|\leqslant 1- \delta$.
\end{defi}

In the same paper he proved that all Lebesgue spaces~$L^p$ are uniformly convex when~$p$ belongs to~$(1, +\infty)$. This statement is an elementary corollary of the following inequalities. Here and in what follows all the norms are the~$L^p$-norms.

\begin{Th}[\textbf{Clarkson inequalities, 1936}]
Let $\phi,\psi \in L^p$. If $p\in [2,+\infty),$ then
$$
2^{p-1} \big( \| \phi \|^p + \| \psi \|^p \big) \geqslant \| \phi + \psi \|^p + \| \phi - \psi \|^p. 
$$
If~$p \in (1,2],$ then
$$
2\big(\|\phi\|^p+\|\psi\|^p\big)^{q/p}\geqslant \|\phi+\psi\|^q+\|\phi-\psi\|^q,
$$
where $q=p/(p-1)$ is the exponent conjugate to~$p$.
\end{Th}

In a time, the question about the dependence of the biggest possible~$\delta$ on~$\eps$  arose. The function~$\delta(\eps)$ is called the modulus of uniform convexity. It appeared that the Clarkson inequality gives the answer to this question only for the case~$p\geqslant 2$, whereas the case~$p<2$ was left open. The sharp dependence~$\delta(\eps)$ had been found by Beurling, who made an oral report about this in Uppsala in 1945. His proof was later written down by Hanner (see~\cite{Ha}).

\begin{Th}[\textbf{Beurling, 1945; Hanner, 1956,  Hanner's inequalities}]\label{HannerINeq}
Let $\phi,\psi \in L^p$. If $p\in [2,+\infty),$ then
$$
\big(\|\phi\|+\|\psi\|\big)^p+\big|\|\phi\|-\|\psi\|\big|^p \geqslant \|\phi+\psi\|^p+\|\phi-\psi\|^p. 
$$
If~$p \in [1,2],$ then
$$
\big(\|\phi\|+\|\psi\|\big)^p+\big|\|\phi\|-\|\psi\|\big|^p \leqslant \|\phi+\psi\|^p+\|\phi-\psi\|^p. 
$$
\end{Th}

With these inequalities at hand, it is easy (see \cite{Ha}) to obtain an estimate for $\delta(\eps)$, which turns out to be sharp.

\begin{Th}\label{ClarksonBeurlingHanner}
\textbf{\textup{1) (Clarkson, 1936).}} If $p\in[2,+\infty),$ then the sharp constant~$\delta(\eps)$ for~$\eps\leqslant 2$ is given by the equality
$$
\delta(\eps)= 1- (1-(\eps/2)^p)^{1/p}.
$$
\textbf{\textup{2) (Beurling, 1945; Hanner, 1956)}}
 If~$p\in(1,2],$ then the sharp constant~$\delta(\eps)$ for $\eps\leqslant 2$ is given by the equality
$$
(1-\delta+\eps/2)^p+|1-\delta-\eps/2|^p=2.
$$
\end{Th}

Beurling's proof, given in~\cite{Ha}, is elementary and brilliant. Its main difficulty, in our opinion, is hidden in the magic inequalities presented in Theorem~\ref{HannerINeq} that are used as a black box, without any explanation of their origin. The purpose of the present article is to show, using the Bellman function method, how can one obtain the answer without ``magic guesses'', but following easy and natural geometric considerations.


The idea of application of optimal control methods to the problems lying at the intersection of mathematical analysis and probability belongs to Burkholder. In his groundbreaking paper~\cite{Burkholder}, Burkholder used these ideas to compute the norm of a martingale transform. Nazarov, Treil, and Volberg brought similar methods (already named the Bellman function) to harmonic analysis (see~\cite{NTV} for the historical rewiew). The article~\cite{Vasyunin} of Vasyunin on computation of sharp constants in the reverse H\"older inequality for Muckenhoupt classes initiated calculations of exact Bellman functions for problems in harmonic analysis. Starting with~\cite{SlavinVasyunin}, the method began to obtain theoretical basis (yet on the basic example of inequalities on the~$\mathrm{BMO}$-space). In~\cite{IOSVZ}, the authors developed the Bellman function theory that unifies rather wide class of problems (see also short report~\cite{SportReport}). It became clear that the computation of Bellman functions is not only analytic and algebraic problem. The geometry of the Bellman function graph also plays an important role (its convexity, the torsion of the boundary value curve, etc.).

\section{Bellman function method}

\subsection{Setting}
All infinite-dimensional~$L^p$-spaces are finitely representable in each other (see~\cite{DJT}, Theorem~$3.2$). Therefore, the moduli of uniform convexity are equal for them. We are going to discuss the uniform convexity of~$L^p([0,1])$ for $p \in (1,+\infty)$. Consider a bit more general problem, namely, estimate the maximum of $\|\phi+\psi\|$ with $\|\phi\|, \|\psi\|, \|\phi-\psi\|$ being fixed; here $\phi,\psi \in L^p$. For a fixed point~$x=(x_1,x_2,x_3) \in \mathbb{R}^3$ consider the  set
$$T(x)= \{(\phi,\psi)\in L^p\times L^p\colon \|\phi\|^p=x_1, \|\psi\|^p=x_2, \|\phi-\psi\|^p=x_3\}.$$
We define the Bellman function $\BB$ by formula
$$
\BB(x)=\sup\{\|\phi+\psi\|^p\colon (\phi,\psi)\in T(x)\}.
$$
Note that $T(x)$ is non-empty if and only if $x_1,x_2,x_3 \geqslant 0$ and the triple $(x_1^{\frac{1}{p}},x_2^{\frac{1}{p}},x_3^{\frac{1}{p}})$ satisfies the triangle inequality. Thus, the natural domain of $\BB$ is a closed cone
$$
\Omg=\{(x_1,x_2,x_3)\in \mathbb{R}^3\colon x_1,x_2,x_3 \geqslant 0,  (x_1^{\frac{1}{p}},x_2^{\frac{1}{p}},x_3^{\frac{1}{p}})\,\, \mbox{satisfies the triangle inequality}\}.
$$

It follows from the very definition that $\BB$ is homogeneous of order one:
$
\BB(k x)=k \BB(x)
$ 
for any $k\geqslant 0$ and $x \in \Omg$.

Note that the value of $\BB$ on the boundary of $\Omg$ can be calculated with ease. Indeed, if $\phi, \psi \in L^p$ and $x=(\|\phi\|^p,\|\psi\|^p,\|\phi-\psi\|^p)\in \partial\Omg$, then the Minkowski inequality for the functions $\phi, \psi, \phi-\psi$ is reduced to an equality. Three possibilities may occur:
\begin{itemize}
\item[1)] $x_1^{\frac1p}=x_2^{\frac1p}+x_3^{\frac1p}$, в этом случае $\BB(x)=\big(x_1^{\frac1p}+x_2^{\frac1p}\big)^p$;
\item[2)] $x_2^{\frac1p}=x_1^{\frac1p}+x_3^{\frac1p}$, в этом случае $\BB(x)=\big(x_1^{\frac1p}+x_2^{\frac1p}\big)^p$;
\item[3)] $x_3^{\frac1p}=x_1^{\frac1p}+x_2^{\frac1p}$, в этом случае $\BB(x)=\big|x_1^{\frac1p}-x_2^{\frac1p}\big|^p$.
\end{itemize}

\subsection{Properties of the Bellman function}
One of the main properties of the Bellman function~$\BB$ is its concavity.

\begin{prop}
The function $\BB$ is concave on $\Omg$.
\end{prop}
\begin{proof}
We have to prove that for any two points $x^{(1)},x^{(2)} \in \Omg$ and any $\alpha \in (0,1)$ the inequality
$$\BB(\alpha x^{(1)}+(1-\alpha)x^{(2)})\geqslant \alpha \BB(x^{(1)})+(1-\alpha)\BB(x^{(2)})$$
is fulfiled.
For any $\theta>0$ and $i=1,2$ we find a pair of functions $(\phi_i, \psi_i) \in T(x^{(i)})$ such that $\|\phi_i+\psi_i\|^p\geqslant \BB(x^{(i)})-\theta$. Consider the concatenation $\phi$ of the functions $\phi_1$ and $\phi_2$ with the weights $\alpha$ and $1-\alpha$ correspondingly, i.e. a function
$$
\phi(t)=
 \begin{cases}
  \phi_1(\frac{t}{\alpha}),& t \in [0,\alpha];\\
  \phi_2(\frac{t-\alpha}{1-\alpha}),& t \in (\alpha,1].
  \end{cases}
$$
We define the concatenation $\psi$ of the functions $\psi_1$ and $\psi_2$ with the weights $\alpha$ and $1-\alpha$ correspondingly in a similar way. Clearly, $(\phi,\psi)\in T(\alpha x^{(1)}+(1-\alpha)x^{(2)})$. Consequently, 
$$
\BB(\alpha x^{(1)}+(1-\alpha)x^{(2)})\geqslant \|\phi+\psi\|^p=\alpha \|\phi_1+\psi_1\|^p+(1-\alpha)\|\phi_2+\psi_2\|^p
\geqslant \alpha \BB(x^{(1)})+(1-\alpha)\BB(x^{(2)})-\theta.
$$
The number $\theta$  was arbitrary, so we get the desired concavity of $\BB$.
\end{proof}

The fucntion $\BB$ occurs to be the minimal function of the class of concave on $\Omg$ functions with given boundary conditions.

\begin{prop}
If $G \colon \Omg \to \mathbb{R}$ is a concave function and $G(x)\geqslant \BB(x)$ for all $x \in \partial\Omg,$ then $G(x)\geqslant \BB(x)$ for all $x \in \Omg$.
\end{prop}
\begin{proof}
Fix any point $x \in \Omg$ and an arbitrary pair of functions $(\phi,\psi) \in T(x)$. Then by Jensen's inequality
\begin{align*}
G(x)=&G\left(\int_0^1 |\phi(t)|^p dt,\int_0^1 |\psi(t)|^p dt,\int_0^1 |\phi(t)-\psi(t)|^p dt\right)\\
\geqslant&\int_0^1 G\left(|\phi(t)|^p, |\psi(t)|^p,|\phi(t)-\psi(t)|^p\right) dt\\ 
\geqslant& \int_0^1 \BB\left(|\phi(t)|^p, |\psi(t)|^p,|\phi(t)-\psi(t)|^p\right) dt\\
=&\int_0^1 |\phi(t)+\psi(t)|^p dt.
\end{align*}
Taking the supremum over all the pairs $(\phi,\psi) \in T(x)$, we obtain the inequality $G(x)\geqslant \BB(x)$.
\end{proof}

Thus, $\BB$ is the minimal among concave on $\Omg$ functions with fixed boundary conditions.

\subsection{Reduction of dimension}
The homogeneity of $\BB$ allows us to reduce the dimension of the problem. 
\begin{rem}\label{rem1}
Let $C$ be a convex cone in $\mathbb{R}^3$ with the vertex at zero. Let $L$ be a plane in $\mathbb{R}^3$ such that for any non-zero $x \in C$ there exists $k>0$ such that $k x \in L\cap C$. Let $G\colon C \to \mathbb{R}$ be a function that is homogeneous of order one. In such a case\textup, the concavity of $G$ on $C$ is equivalent to the concavity of $G$ on $C\cap L$.
\end{rem}
\begin{proof}
Obviously, if $G$ is concave on $C$, then it is also concave on $C\cap L$. Let us prove the converse. 

Suppose that $x_1, x_2 \in C$, $\alpha \in (0,1)$, and $x=\alpha x_1+ (1-\alpha) x_2$. Find the numbers $k,k_1,k_2>0$ such that $kx, k_1x_1, k_2x_2 \in L$. Note that $kx=\alpha\frac{k}{k_1}k_1x_1+(1-\alpha)\frac{k}{k_2}k_2x_2$. Using the concavity of $G$ on $L\cap C$, we get:
$$
G(xk)\geqslant \alpha\frac{k}{k_1}G(k_1x_1)+(1-\alpha)\frac{k}{k_2}G(k_2x_2).
$$
The first order homogeneity of $G$ leads to the wanted inequality
$$
G(x)\geqslant \alpha G(x_1)+(1-\alpha)G(x_2).
$$
\end{proof}

The role of the cone $C$ in our case is played by $\Omg$, the plane $\{x \in \mathbb{R}^3\colon x_1+x_2+x_3=1\}$ stands for $L$. By virtue of Remark~\ref{rem1}, the restriction of $\BB$ to $\Omg\cap L$ is a concave function, and moreover, the minimal among all the concave functions with the same boundary conditions on $\partial(\Omg\cap L)$.

Thus, the initial three-dimensional problem concerning the minimal concave function is reduced to the two-dimensional problem that looks like this.
Consider a convex set
\begin{equation}\label{Omega}
\Omega=\{(x_1,x_2)\in \mathbb{R}^2\colon (x_1,x_2,1-x_1-x_2)\in \Omg\},
\end{equation}
which is the projection of $\Omg\cap L$, and a function 
\begin{equation}\label{funB}
\B(x_1,x_2)=\BB(x_1,x_2,1-x_1-x_2)
\end{equation}
on it. 
The function $\B$ is concave on $\Omega$ and is the minimal among concave functions with fixed boundary values. In other words, if $G \colon \Omega \to \mathbb{R}$ is concave and $G\geqslant \B$ on $\partial\Omega$, then $G\geqslant \B$ on the whole domain $\Omega$.

We write down the values of $\B$ on $\partial\Omega$. The boundary $\partial\Omega$ consists of three parts that match three cases of degeneration in the triangle inequality. Namely, $\partial\Omega = \gamma^{[1]} \cup \gamma^{[2]}\cup \gamma^{[3]}$, where
\begin{align}
\label{gamma1} \gamma^{[1]}(s)&=\left(\frac{1}{s^p+(1-s)^p+1}, \frac{s^p}{s^p+(1-s)^p+1}\right), \quad &s \in [0,1];\\
\label{gamma2} \gamma^{[2]}(s)&=\left(\frac{(1-s)^p}{s^p+(1-s)^p+1}, \frac{1}{s^p+(1-s)^p+1}\right), \quad &s \in [0,1];\\
\label{gamma3} \gamma^{[3]}(s)&=\left(\frac{s^p}{s^p+(1-s)^p+1}, \frac{(1-s)^p}{s^p+(1-s)^p+1}\right), \quad &s \in [0,1].
\end{align}
The values of $\B$ on $\partial\Omega$ are given by the following equalities:
\begin{equation}\label{bv}
\B\left(\gamma^{[1]}(s)\right)=\frac{(1+s)^p}{s^p+(1-s)^p+1};\,
\B\left(\gamma^{[2]}(s)\right)=\frac{(2-s)^p}{s^p+(1-s)^p+1};\,
\B\left(\gamma^{[3]}(s)\right)=\frac{|1-2s|^p}{s^p+(1-s)^p+1}.
\end{equation}

\section{Minimal concave functions on convex compact sets}
In this section we discuss some properties of minimal concave functions on convex compact sets. Let~$\omega \subset \mathbb{R}^d$ be strictly convex compact set with non-empty interior (by strict convexity we mean that the boundary does not contain segments). Let $f \colon \partial \omega \to \mathbb{R}$ be a fixed continuous function. By the symbol $\Lambda_{\omega,f}$ we denote the set of all convex on $\omega$ functions~$G$ such that $G(x)\geqslant f(x)$ for all~$x \in \partial\omega$. Define for $x\in \omega$ the pointwise infimum by formula
$$
\Bg_{\omega,f}(x)=\inf\{G(x)\colon G\in \Lambda_{\omega,f}\}.
$$
Obviously, $\Bg_{\omega,f}\in \Lambda_{\omega,f}$, therefore, $\Bg_{\omega,f}$ is the minimal concave on $\omega$ function that majorizes $f$ on~$\partial\omega$. Note that $\Bg_{\omega,f}=f$ on $\partial\omega$, because in the opposite case we could have diminished the value~$\Bg_{\omega,f}$ on $\partial\omega$ keeping the concavity. This would have contradict the minimality.

The concavity of a function is equivalent to the convexity of its subgraph. The pointiwse minimality is equivalent to the minimality by inclusion of the subgraph. These simple considerations lead to the following conclusion.

\begin{prop}\label{subgraph}
Let
$$
\Sg(f)=\{(x,y) \in \partial\omega\times\mathbb{R}\colon y \leqslant f(x)\},\quad  \Sg(\Bg_{\omega,f})=\{(x,y) \in \omega\times\mathbb{R}\colon y \leqslant \Bg_{\omega,f}(x)\}
$$
be the subgraphs of $f$ and $\Bg_{\omega,f}$ correspondingly. In such a case\textup, 
$\Sg(\Bg_{\omega,f})=\conv(\Sg(f)),$ where $\conv$ stands for the convex hull.
\end{prop}
\begin{proof}
We note that the subgraph $\Sg(\Bg_{\omega,f})$ of a concave function $\Bg_{\omega,f}$ is a convex set, $ \Bg_{\omega,f}\geqslant f$ on~$\partial \omega$, thus $\Sg(\Bg_{\omega,f}) \supset \conv(\Sg(f))$. 

The function $f$ is continuous, $\omega$ is compact, therefore, the set $\conv(\Sg(f))$ is closed. We define the function $G$ on $\omega$ in such a way that its subgraph $\Sg(G)$ coincides with $\conv(\Sg(f))$. Clearly,~$G \in \Lambda_{\omega,f}$, as a result, $G \geqslant \Bg_{\omega,f}$ on $\omega$. But then $\Sg(\Bg_{\omega,f})\subset \Sg(G)=\conv(\Sg(f))$.
\end{proof}

\begin{cor}\label{fol}
For any point $x_0 \in \omega$ there exists a number $k \leqslant d+1$ and points $x_1, \dots, x_k \in \partial\omega$ such that
$x_0 \in \conv(x_1, \dots, x_k),$ and the function  $\Bg_{\omega,f}$ is linear on $\conv(x_1, \dots, x_k)$.
\end{cor}
\begin{proof}
Note that the case where $x_0 \in \partial \omega$ is trivial. In the remaining cases, $x_0\in \intt(\omega)$. Let~$P_0=(x_0,\Bg_{\omega,f}(x_0))$. Due to Proposition~\ref{subgraph} we have $P_0\in \Sg(\Bg_{\omega,f})=\conv(\Sg(f))$, therefore, by the Carath\'eodory theorem about convex hull, $P_0$ belongs to the convex hull of not more than $d+2$ points of the set $\Sg(f)$. We note that $P_0 \in \partial \Sg(\Bg_{\omega,f})$, consequently, $P_0$ cannot lie inside the interior of the convex hull of $d+2$ points belonging to $\Sg(f)$. Therefore, there exists $k \leqslant d+1$ and the points $P_i=(x_i, y_i) \in \Sg(f), i=1,\dots, k$, such that $P_0 \in \conv(P_1, \dots, P_k)$. We may assume that the number $k$ is the smallest possible, i.e. for any $k'<k$ the point $P_0$ does not lie inside convex hull of any $k'$ points belonging to $\Sg(f)$. In such a case, there exist numbers~$\alpha_1, \dots, \alpha_k \in (0,1)$ such that $\sum \alpha_i=1$ and $P_0=\sum_{i=1}^k \alpha_i P_i$. We note that the function $\Bg_{\omega,f}$ is concave on $\conv(x_1,\dots,x_k)$, $\Bg_{\omega,f}(x_i)\geqslant f(x_i)\geqslant y_i$, but $\Bg_{\omega,f}(\sum_{i=1}^k\alpha_i x_i)=\sum_{i=1}^k\alpha_i y_i$. The numbers $\alpha_i$ are positive, thus, $\Bg_{\omega,f}(x_i)=f(x_i)=y_i$ for all $i =1, \dots, k$, and the function $\Bg_{\omega,f}$ is linear on $\conv(x_1, \dots, x_k)$. 
\end{proof}

\section{Torsion and foliation}
We return to the domain $\Omega$ in $\mathbb{R}^2$ defined by equality~\eqref{Omega}. Let $F\colon\partial\Omega \to \mathbb{R}$ be the restriction of~$\B$ to~$\partial\Omega$, given by formula~\eqref{funB}. Formula~\eqref{bv} together with formulas~\eqref{gamma1}, \eqref{gamma2}, and \eqref{gamma3} defines the function $F$ directly. We note that the function $F$ is continuous on $\partial\Omega$.
 With the notation of the previous section, $\B=\Bg_{\Omega,F}$. 

Direct computations show that when $p \in (1, +\infty)$, the piecewise parametrization~\eqref{gamma1},~\eqref{gamma2},~\eqref{gamma3} of the boundary $\partial\Omega$ appears to be $C^1$-smooth. Moreover, the function $F$, defined on $\partial\Omega$, is also $C^1$-smooth in this parametrization.

If $p=2$, then the function $F$ is nothing but the restriction of a linear function to $\partial\Omega$. Therefore, the function $\B$ is linear. In the case where $p\ne 2$ the situation is more complicated. By Corollary~\ref{fol}, the whole set $\Omega$ is covered by triangles and segments (in what follows, we call such segments chords) whose endpoints lie on $\partial\Omega$, on each of which the function $\B$ is linear. Our aim is to understand how is this covering arranged. The following key lemma will help us with this business (for the needed stuff from differential geometry, see, e.g. \cite{Pog}).

\begin{lem}\label{tor+-}
Let $\omega\subset \mathbb{R}^2$ be a strictly convex closed set. Let $a_1, a_2 \in \partial \omega,$ and the tangents to~$\omega$ at the points $a_1$ and $a_2$ intersect at the point $b$. Let $\I\subset \mathbb{R}$ be some open interval\textup{,} $\gamma \colon \I \to \partial \omega$ a parametrization of the part of $\partial\omega$ that contains the arc between $a_1$ and $a_2$ lying inside the triangle~$a_1ba_2$. Suppose $t_1,t_2 \in \I$ are such that $\gamma(t_i)=a_i,$ $i=1,2$. Assume that the curve $\gamma$ goes along $\partial\omega$ in the counter-clockwise direction and $t_2>t_1$.

Let $G$ be a concave function on $\omega,$ linear on the segment connecting $a_1$ and $a_2$. Let the curve~$(\gamma, G(\gamma))$ belong to  $C^1$ on $\I$. In such a case\textup, non of the following conditions may hold\textup{:}
\begin{itemize}
\item[1\textup)] the curve $(\gamma, G(\gamma))$ belongs to $C^3$ on $\I,$ its torsion is positive on $(t_1,t_2);$
\item[2\textup)] the curve $(\gamma, G(\gamma))$ belongs to $C^3$ on $\I,$ its torsion is negative on $(t_1,t_2);$
\item[3\textup)] there exists $t_0 \in (t_1,t_2)$ such that the curve $(\gamma, G(\gamma))$ belongs to the class $C^3$ on $\I\setminus\{t_0\},$ its torsion is negative on $(t_1,t_0)$ and positive on $(t_0,t_2)$.
\end{itemize}
\end{lem}
\begin{proof}
Turn the first two coordinates and make a reparametrization, if needed, in such a fashion that the condition $\gamma_1'(t)>0$ when $t \in [t_1,t_2]$, where $\gamma=(\gamma_1,\gamma_2)$, is satisfied. The conditions of the lemma does not change when the domain and the curve undergo such transformations. 

Due to concavity of~$G$ on~$\omega$, we can find a linear function~$L \colon \mathbb{R}^2 \to \mathbb{R}$ such that $G \leqslant L$ on $\omega$ and $G=L$ on the segment $[a_1,a_2]$. Without loss of generality, we may assume that $L\equiv 0$ (if it not so, we may consider the function $G-L$ instead  of $G$, keeping the conditions of the lemma).

We introduce the notation $f(t)=G(\gamma(t)), v(t)=\frac{\gamma_2'(t)}{\gamma_1'(t)},  u(t)=\frac{f'(t)}{\gamma_1'(t)}$. The function $v$ is increasing due to the convexity of $\omega$. Direct computations show that the sign of the torsion of the curve $(\gamma(t),f(t))$ defines the convexity (or concavity) of the curve $(v(t),u(t))$:
\begin{equation}\label{tor}
u''v'-v''u'=\frac{1}{(\gamma_1')^3}
\begin{vmatrix}
\gamma_1'&\gamma_2'& f'\\
\gamma_1''&\gamma_2''&f''\\
\gamma_1'''&\gamma_2'''&f'''\\
\end{vmatrix}.
\end{equation}

The function $f$ defined on the interval $\I$ satisfies the inequality $f\leqslant 0$ and the equalities $f(t_1)=f(t_2)=0$. Thus, $f'(t_1)=f'(t_2)=0$ and $f''(t_1)\leqslant 0$, $f''(t_2)\leqslant 0$. It follows that 
\begin{equation}\label{eq5}
u(t_1)=u(t_2)=0, \quad u'(t_1) \leqslant 0, \quad u'(t_2)\leqslant 0.
\end{equation}

Now we treat each of the three cases by itself. In the first case, the curve $(\gamma,f)$ belongs to the class $C^3$ on $\I$ and its torsion is positive on $(t_1,t_2)$. Consequently, following formula~\eqref{tor}, the curve~$(v(t),u(t))$ must be strictly convex when $t \in (t_1,t_2)$. But this contradicts conditions~\eqref{eq5}. Similarly, in the second case the curve must be strictly concave, which also contradicts condition~\eqref{eq5}.

In the third case the curve $(v(t),u(t))$ is strictly concave when $t \in (t_1,t_0)$, $u(t_1)=0\geqslant u'(t_1)$, therefore, $u(t_0)<0$. On the other hand, the strict convexity of $(v(t),u(t))$ when $t \in (t_0, t_2)$ and the conditions $u(t_2)=0\geqslant u'(t_2)$ lead to the inequality $u(t_0)>0$, a contradiction. The lemma is proved.

\end{proof}

We are going to apply Lemma~\ref{tor+-} to the concave function $\B$ defined on a strictly convex set $\Omega$ in order to get an idea how can the chords be arranged. We have to compute the torsions $\tau^{[i]}(s)$ of the curves $(\gamma^{[i]}(s),\B(\gamma^{[i]}(s)))$:
\begin{align*}
\tau^{[1]}(s)&=-\frac{2( p-2)( p-1)^2 p^3((1-s)s(1+s))^{p-3}}{(s^p+( 1-s)^p+1)^4}, \quad &s\in(0,1);\\
\tau^{[2]}(s)&=\frac{2( p-2)( p-1)^2 p^3((1-s)s(2-s))^{p-3}}{(s^p+( 1-s)^p+1)^4}, \quad &s\in(0,1);\\
\tau^{[3]}(s)&=-\textup{sign}(1-2s)\frac{2( p-2)( p-1)^2 p^3((1-s)s|1-2s|)^{p-3}}{(s^p+( 1-s)^p+1)^4}, \quad &s\in\left(0,\frac{1}{2}\right)\cup\left(\frac{1}{2},1\right);
\end{align*}

These formulas show us the torsion signs of the graph of $F$ on $\partial\Omega$. When $p>2$, the inequalities $\tau^{[1]}(s)<0$, $\tau^{[2]}(s)>0$ when $s \in (0,1)$, $\tau^{[3]}(s)<0$ when $s \in (0,\frac{1}{2})$, and  $\tau^{[3]}(s)>0$ when $s \in (\frac{1}{2},1)$ hold. When $p<2$ all the signs in these inequalities are changed for opposite ones. The domain $\Omega$, the signs of the torsions of the corresponding curves, and the points where these torsions change sign are indicated on~Figure~\ref{torsign}
 
\begin{figure}[h!]
\includegraphics[height=6.5cm]{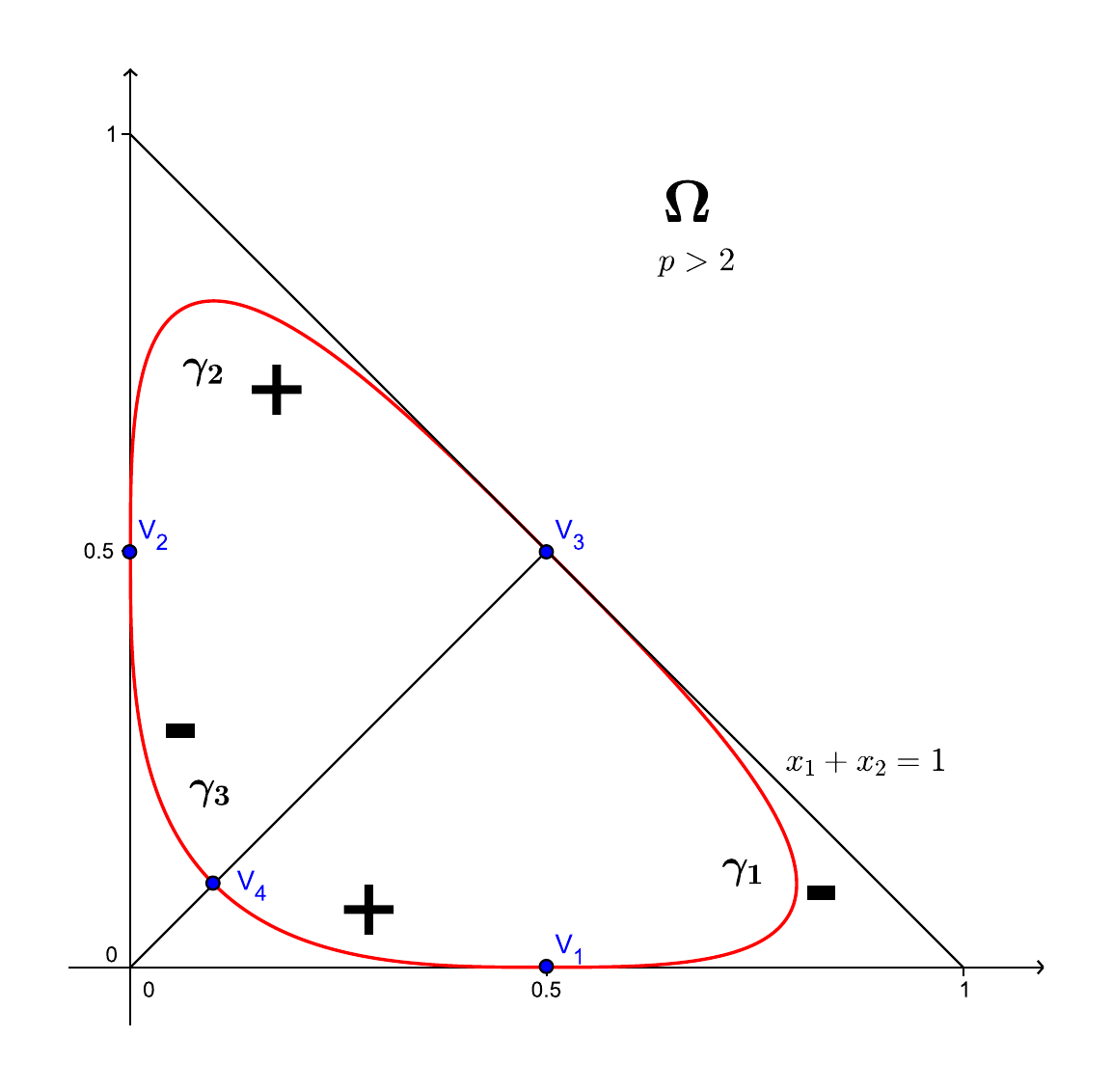}
\includegraphics[height=6.5cm]{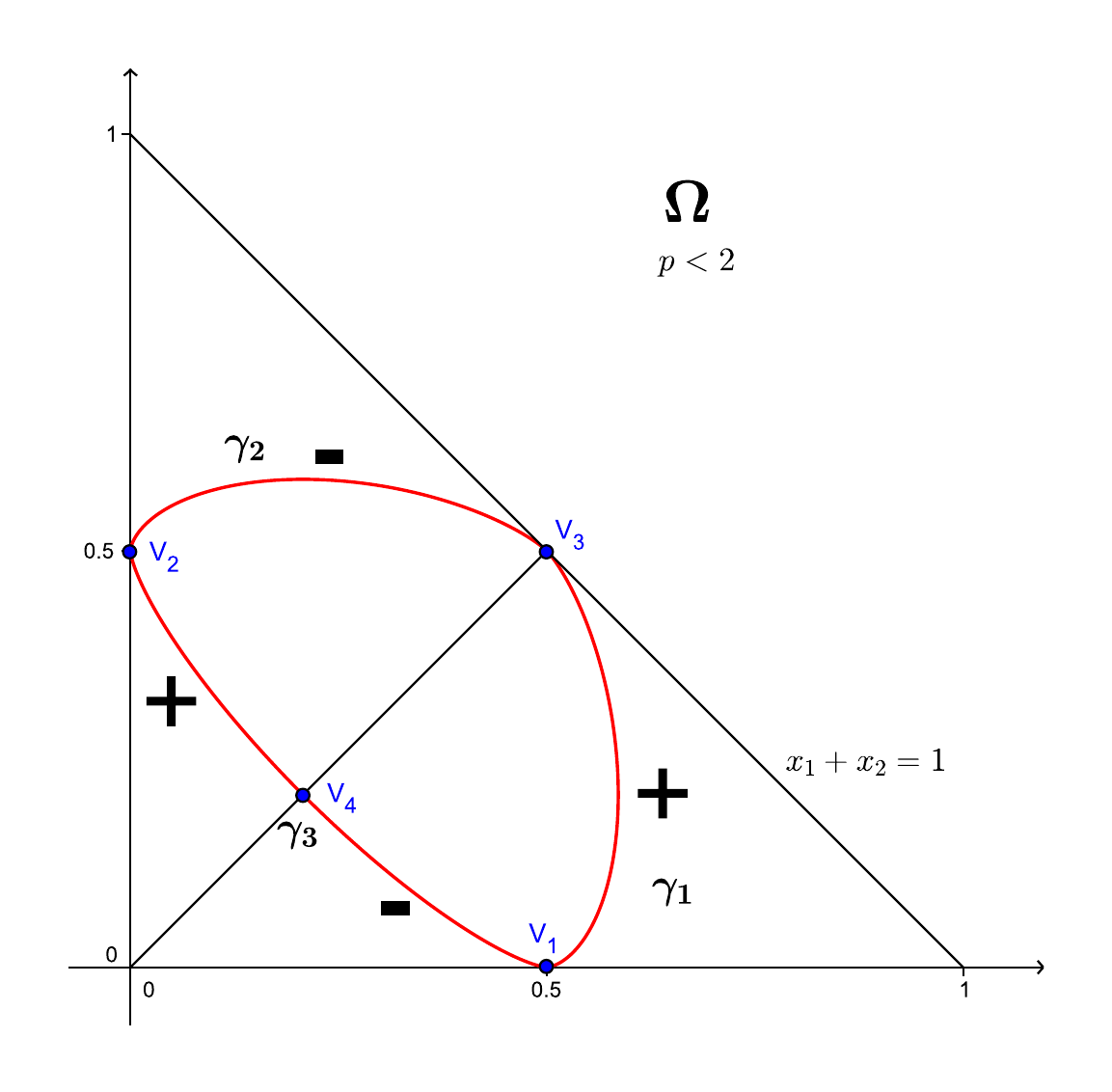}
\caption{Domain $\Omega$ and signs of torsion.}
\label{torsign}
\end{figure}
 
The following remark is an easy, but important addition to Lemma~\ref{tor+-}.
\begin{rem}
Let $\I$ be a segment with the endpoints on $\partial\Omega$ such that the function $\B$ is linear on it. In such a case\textup, for any $\rho>0$ and any of the two closed arcs of $\partial\Omega$ subtended by $\I,$ there exists a segment $\I_1$ with the endpoints on this arc such that $0<|\I_1|<\rho$ and the function $\B$ is linear on $\I_1$.
\end{rem}
\begin{proof}
Assume the contrary, let the statement be incorrect for one of the closed arcs subtended by~$\I$. Take the chord~$\I_1$ with the endpoints on this arc in such a way that the function~$\B$ is linear on it and the chord~$\I_1$ subtends the shortest arc with this property (such an arc exists because the domain~$\Omega$ is compact and the function~$\B$ is continuous). Take any point $x_0 \in \intt(\Omega)$ that is separated by~$\I_1$ from~$\I$. By Corollary~\ref{fol}, we can find a segment or a triangle such that the function~$\B$ is linear on it, its endpoints lie on $\partial\Omega$, and it contains~$x_0$. Due to the minimality of the arc subtended by~$\I_1$, this segment or triangle intersects the chord~$\I_1$ in its interior points. Thus, we have found the chord~$\I_2$ such that the function~$\B$ is linear on it and $\I_1\cap \I_2\cap \intt(\Omega)\ne \varnothing$. But in such a case, the function $\B$ has to be linear on $\conv(\I_1\cup \I_2)$. This also allows us to find a chord subtending a shorter arc than~$\I_1$ such that~$\B$ is linear on it. A contradiction.
\end{proof}

Together with Lemma~\ref{tor+-}, this remark implies the following Corollary.
\begin{cor}
Suppose $\I$ is a segment such that its endpoints are on $\partial\Omega$ and the function $\B$ is linear on it. In such a case\textup, there exists a point where the torsion of the graph of~$F$ changes its sign from~$+$ to~$-$ \textup(in the  counter-clockwise orientation\textup) on both sides of~$\I$. 
\end{cor}
 
By the fact that when $p\ne 2$ there exists only two changes of sign of the torsion from $+$ to $-$, there does not exist a triangle such that its endpoints belong to $\partial \Omega$ and the function $\B$ is linear on it. Consequently, a chord such that its endpoints lie on~$\partial\Omega$ and~$B$ is linear on it passes through each point of $\Omega$. Moreover, these chords cannot intersect each other in interior points, because in such a case the function $\B$ would have been linear on the convex hull of such intersecting chords. Such a tiling of $\Omega$ by these disjoint chords is called a \emph{foliation}.

We are going to use the symmetry of our problem. The set  $\Omega$ and the boundary function $F$ do not change when the first two coordinates are permuted. So, the function $\B$ and the foliation also have this property. Pick a point $x\in \Omega$ such that $x_1=x_2$ and find the chord that contains it. By symmetry, this chord intersects the symmetric chord, so it is symmetric to itself. Thus, this chord either lies on the symmetry axis or is orthogonal to it. From both sides of this chord there are the points where the torsion changes sign from $+$ to $-$. As a result, we see that when $p>2$ the chord lies on the symmetry axis, and when $p<2$, it is perpendicular to this axis. This justifies Figure~\ref{foliation}, we have proved that the chords of the foliation are arranged as its is drawn there.
\begin{figure}[h!]
\includegraphics[height=6.5cm]{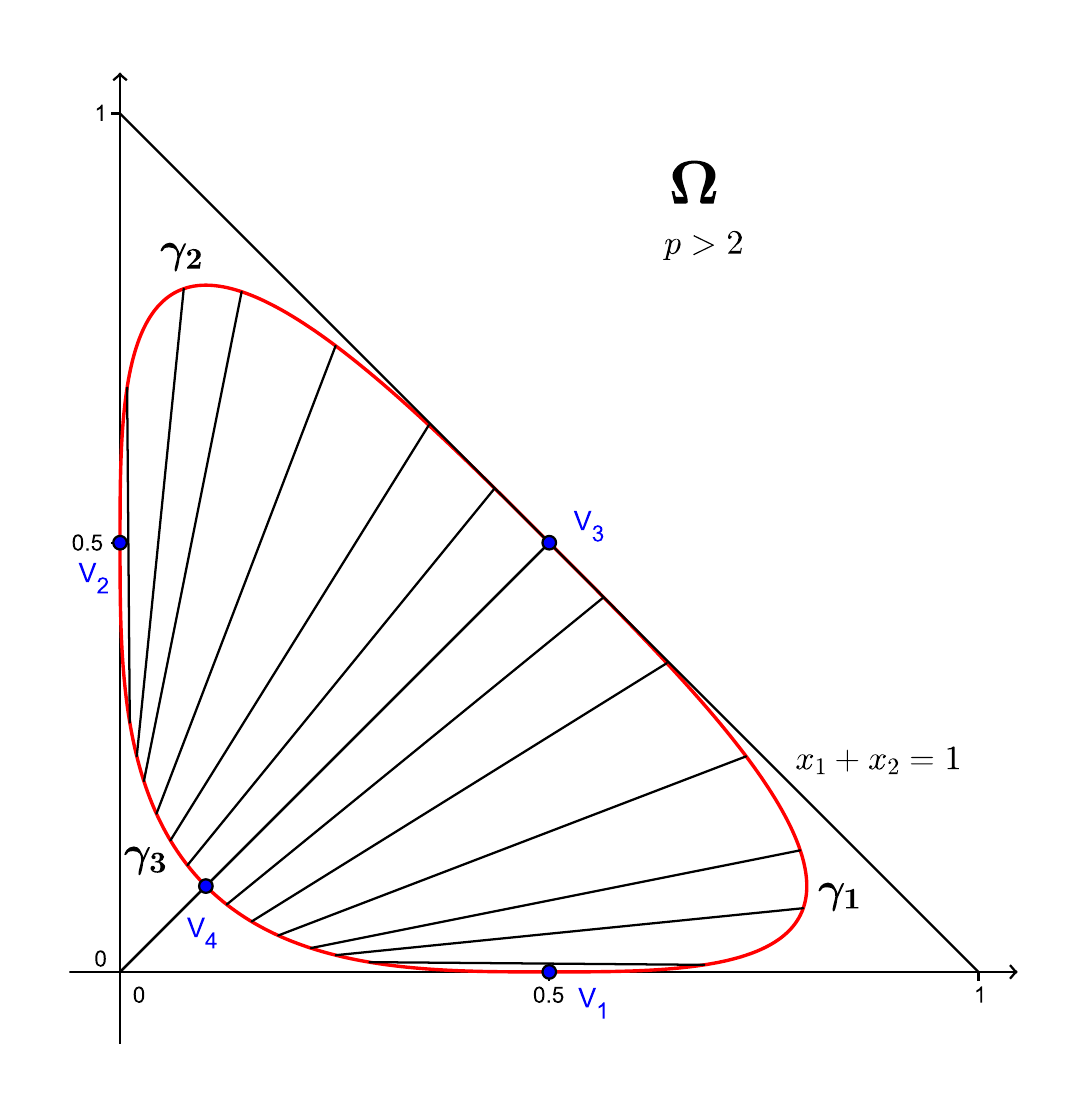}
\includegraphics[height=6.5cm]{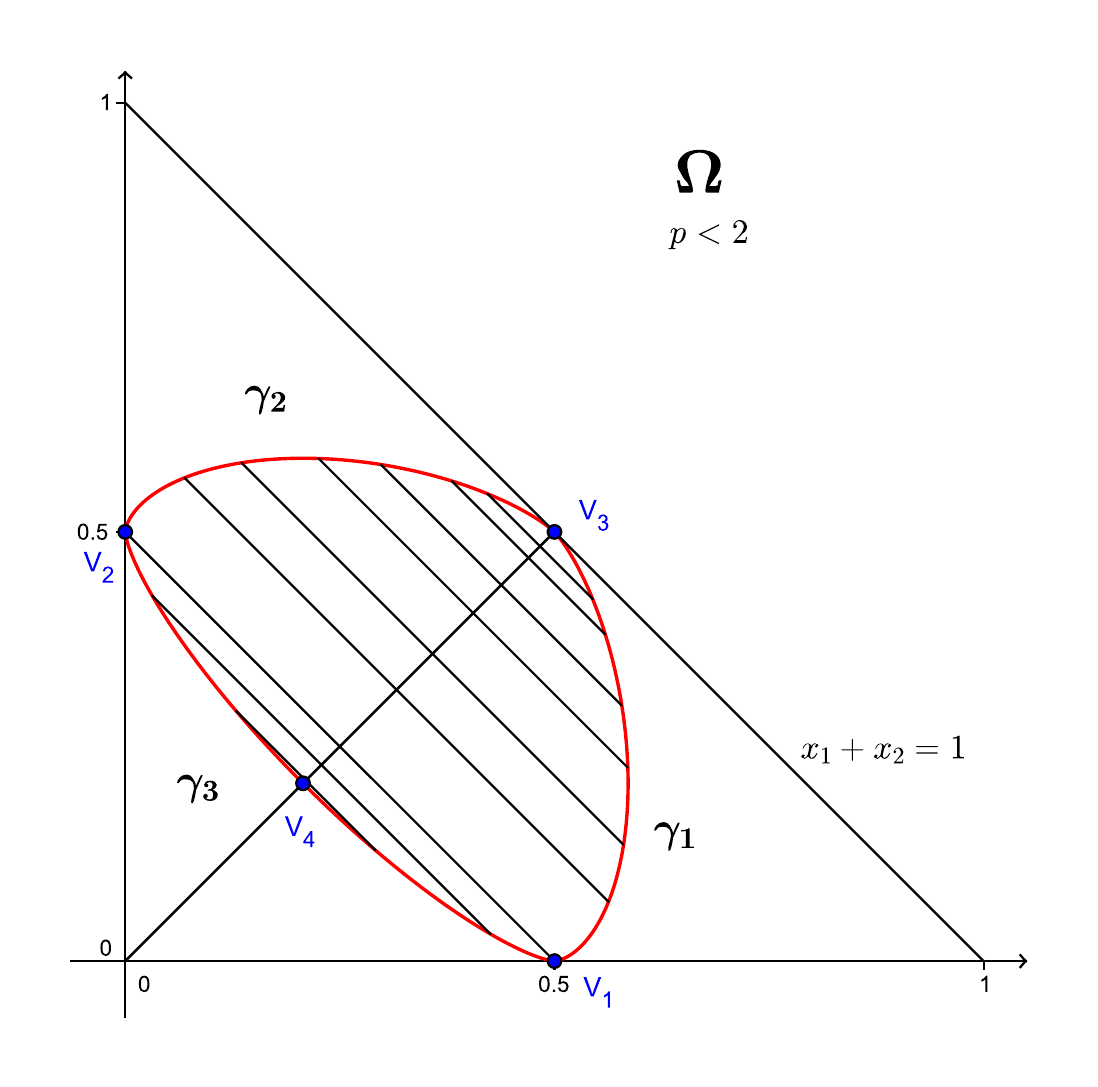}
\caption{Фолиация.}
\label{foliation}
\end{figure}
  
\section{Computations and answer}
Now we are equipped enough to calculate the values of $\B$ on the line $x_1=x_2$. In the case $p>2$, the function in question is linear on the said line. Therefore, using the boundary conditions~\eqref{bv}, we see that $\B(x_1,x_1)=(2+2^p)x_1-1$. Returning to the homogeneous function $\BB$, we obtain
$$
\BB(1,1,t)=(t+2)\BB\left(\frac{1}{t+2},\frac{1}{t+2},\frac{t}{t+2}\right)=(t+2)\B\left(\frac{1}{t+2},\frac{1}{t+2}\right)=2^p-t,
$$
which, in its turn, leads to the desired modulus of uniform convexity $\delta(\eps)$ via the equation
$$
2^p(1-\delta(\eps))^p=\sup\limits_{t\in [\eps^p,2^p]}\BB(1,1,t)=2^p-\eps^p.
$$

In the case where $p<2$, it is impossible to express the values of $\B$ on the line $x_1=x_2$ by an elementary formula. Due to the fact that the function $\B$ is linear on the chord~$c(t)$ that passes through the point $(t,t)$ and is symmetric on it,~$\B$ is constant on this chord~$c(t)$ and coincides with the boundary condition. If $2t\in [\frac{1}{2},1]$, then~$c(t)$ ends at the boundary $\gamma^{[1]}$, therefore, there exists a unique solution $s \in [0,1]$ of the equation $\gamma^{[1]}_1(s)+\gamma^{[1]}_2(s)=2t$, and 
$$
\B(t,t)=F\left(\gamma^{[1]}(s)\right)=\frac{(1+s)^p}{1+s^p+(1-s)^p}.
$$
If $2t \in \left[\frac{1}{2^{p-1}+1},\frac{1}{2}\right]$, then~$c(t)$ ends at the boundary $\gamma^{[3]}$, therefore, there exists a unique solution $s \in [\frac{1}{2},1]$ of the equation $\gamma^{[3]}_1(s)+\gamma^{[3]}_2(s)=2t$, and 
$$
\B(t,t)=F\left(\gamma^{[3]}(s)\right)=\frac{(2s-1)^p}{1+s^p+(1-s)^p}.
$$
As in the previous case, these equalities allow one to find the modulus of uniform convexity $\delta(\eps)$ from the equation
\begin{equation}\label{mono}
2^p(1-\delta(\eps))^p=\sup\limits_{t\in [\eps^p,2^p]}\BB(1,1,t)=\sup\limits_{t\in [\eps^p,2^p]}(t+2)\B\left(\frac{1}{t+2},\frac{1}{t+2}\right)=(\eps^p+2)\B\left(\frac{1}{\eps^p+2},\frac{1}{\eps^p+2}\right).
\end{equation}
The last equality in~\eqref{mono} follows from the monotonicity of the function $\B(t,t)/t$. This monotonicity can be justified with the help of an easy consideration. The function $\diag\colon t \mapsto \B(t,t)$ is defined on the interval~$\left[\frac{1}{2^p+2},\frac{1}{2}\right]$, is non-negative and concave on it, it is zero at the left endpoint. Therefore, the function $\diag(t)/t$ first increases  (till the moment when the tangent at the point~$(t,\diag(t))$ passes through zero), and then decreases. We have to verify that it grows till the point $t=\frac{1}{2}$. It follows from the inequality $\diag(t)/t \leq \diag(\frac{1}{2})/\frac{1}{2}$, which in its turn is equivalent to the inequality $\diag(t) \leq 2^pt$. The latter inequality can be justified with the help of the minimality of $\B$. Using formulas~\eqref{bv} and~\eqref{gamma1},~\eqref{gamma2},~\eqref{gamma3}, it is easy to see that the linear function $G(x_1,x_2)=2^{p-1}(x_1+x_2)$ majorizes $\B$ on $\partial\Omega$ and thus on the whole~$\Omega$.

\section{Further results}
To solve the initial problem, it is enough to calculate the values of the function $\B$ on the symmetry axis. In the case where $p<2$ all the chords are perpendicular to the symmetry axis; this fact allows to compute the values of the function $\B$ at any point. For this purpose, it suffices to find the endpoints of the chord passing through the point in question. In the case where $p>2$, the situation is more complicated. To calculate the values of the function $\B$ off the symmetry axis, one is forced to use additional considerations. The corresponding technique had been partly developed in~\cite{IOSVZ}, later was modified to fit the general situation; it will be set out in a forthcoming paper.

By using similar methods, one can calculate how big can the value~$\|\theta \phi + (1-\theta)\psi\|$ be when~$\|\phi\|$,~$\|\psi\|$ and~$\|\phi-\psi\|$ are fixed (here $\theta$ is some fixed number), or any other ``decent'' function  of~$\phi$ and~$\psi$ (by calculation we mean that the answer can be represented as an implicit function expressing~$\delta$ in terms of~$\eps$, e.g. as in Theorem~\ref{ClarksonBeurlingHanner}). 



\section{Acknowledgment}
We are greatful to N.~K.~Nikolskii, who, in his lecture in P.~L.~Chebyshev Laboratory, attracted our attention to this topic, in particular, to the paper~\cite{BallPisier}, which was the starting point of our studies. We also thank F.~V.~Petrov and~D.~S.~Chelkak for useful comments and remarks.

The last but not least, we are grateful to our teacher V.~I.~Vasyunin.   

%


Paata Ivanisvili

Department of Mathematics, Michigan State University, East Lansing, MI
48823, USA

ivanishvili.paata@gmail.com
\medskip

Dmitriy M. Stolyarov

St. Petersburg Department of Steklov Mathematical Institute, Russian Academy of Sciences (PDMI RAS);
P. L. Chebyshev Research Laboratory, St. Petersburg State University.

dms@pdmi.ras.ru

\medskip
Pavel B. Zatitskiy

St. Petersburg Department of Steklov Mathematical Institute, Russian Academy of Sciences (PDMI RAS);
P. L. Chebyshev Research Laboratory, St. Petersburg State University.

paxa239@yandex.ru
\end{document}